\documentclass[12pt,reqno]{amsart}

\usepackage[latin1]{inputenc}
\usepackage{amsmath}
\usepackage{amsfonts}
\usepackage{amssymb}
\usepackage{graphics}
\usepackage{enumerate}
\usepackage{subfigure}
\usepackage{amssymb,amsmath,amsthm,amscd,epsf,latexsym,verbatim,graphicx,amsfonts}
\input epsf.tex

\usepackage{amscd}
\usepackage{amsmath}
\usepackage{amssymb}
\usepackage{amsthm}
\usepackage{epsf}
\usepackage{latexsym}
\usepackage{verbatim}
\input epsf.tex

\theoremstyle{plain}
\newtheorem{theorem}{Theorem}[section]

\newtheorem{corollary}[theorem]{Corollary}
\newtheorem{lemma}[theorem]{Lemma}
\newtheorem{prop}[theorem]{Proposition}
\theoremstyle{definition}

\theoremstyle{remark}

\newcommand{\nri}{n\rightarrow\infty}
\newcommand{\bbZ}{\mathbb{Z}}
\newcommand{\bbR}{\mathbb{R}}
\newcommand{\bbC}{\mathbb{C}}

\newcommand{\bbD}{\mathbb{D}}
\newcommand{\bbN}{\mathbb{N}}

\newcommand{\mcn}{\mathcal{N}}

\newcommand{\mcp}{\mathcal{P}}
\newcommand{\mcl}{\mathcal{L}}
\newcommand{\mcm}{\mathcal{M}}
\newcommand{\mcr}{\mathcal{R}}

\newcommand{\mt}{\widetilde{M}}
\newcommand{\bard}{\overline{\bbD}}
\newcommand{\barc}{\overline{\bbC}}
\newcommand{\barg}{\overline{G}}
\newcommand{\eitheta}{e^{i\theta}}
\newcommand{\mutil}{\tilde{\mu}}

\DeclareMathOperator*{\supp}{supp}

\DeclareMathOperator*{\wlim}{w-lim}

\title[]{Ratio Asymptotics, Hessenberg Matrices, and Weak Asymptotic Measures}
\author[]{Brian Simanek}
\date{}

\begin{document}
\maketitle

\begin{abstract}
We discuss the relationship between ratio asymptotics for general orthogonal polynomials and the asymptotics of the associated Bergman shift operator.  More specifically, we consider the case in which a measure is supported on an infinite compact subset of the complex plane.  We show that there is a straightforward connection between the corresponding orthonormal polynomials exhibiting ratio asymptotics and the corresponding Bergman shift operator being asymptotically Toeplitz.  We also discuss a connection to the weak asymptotics of the measures derived from the orthonormal polynomials.
\end{abstract}

\vspace{4mm}


\vspace{2mm}


\vspace{2mm}

\normalsize

\section{Introduction}\label{intro}

Let $\mu$ be a finite Borel measure whose support is an infinite and bounded subset of the complex plane $\bbC$.  Given such a measure, one can perform Gram-Schmidt orthogonalization on the sequence $\{1,z,z^2,z^3,\ldots\}$ in the space $L^2(\bbC,\mu)$ to arrive at a sequence of polynomials $\{\varphi_n(z;\mu)\}_{n\geq0}$ satisfying
\[
\int_{\bbC}\varphi_n(z;\mu)\overline{\varphi_m(z;\mu)}d\mu(z)=\delta_{mn},
\]
and normalized so that $\varphi_n$ has positive leading coefficient $\kappa_n$.  The polynomials $\{\varphi_n\}_{n=0}^{\infty}$ are called the orthonormal polynomials for the measure $\mu$.  One is also often lead to consider the monic orthogonal polynomials, which we denote by $\Phi_n(z;\mu)=\kappa_n^{-1}\varphi_n(z;\mu)$.  The polynomial $\Phi_n(z;\mu)$ is a monic polynomial of degree exactly $n$ and also satisfies an orthogonality condition.

One can study the asymptotics of the polynomials $\{\varphi_n\}_{n\in\bbN}$ in several ways.  The most general asymptotic behavior that can be discerned is called \textit{root asymptotics}, which occurs precisely when the limit
\[
\lim_{\nri}|\varphi_n(z)|^{1/n}
\]
exists for appropriate values of $z$.  More refined estimates are necessary to discern \textit{ratio asymptotics}, which concern the existence of
\begin{align}\label{ratdefn}
\lim_{\nri}\frac{\varphi_n(z;\mu)}{\varphi_{n+1}(z;\mu)}
\end{align}
for appropriate values of $z$.  Many examples of measures whose orthonormal polynomials exhibit ratio asymptotics can be found in \cite{SimBlob,SimaRat,Suetin}.

One usually obtains root and ratio asymptotics for $z$ outside the polynomial convex hull of the support of the measure.  It is often more difficult to investigate the behavior of the orthonormal polynomials inside the support of the measure.  One common way to approach this problem is to look at weak limits of the measures $\{|\varphi_n(z;\mu)|^2d\mu(z)\}_{n\in\bbN}$.  The polynomial $\Phi_n(\cdot;\mu)$ satisfies the \textit{extremal property}:
\[
\|\Phi_n(\cdot;\mu)\|_{L^2(\mu)}=\inf\{\|Q\|_{L^2(\mu)}:Q(z)=z^n+\mbox{ lower order terms}\}.
\]
Therefore, one expects the polynomial $\Phi_n(z;\mu)$ to be small where the measure $\mu$ is dense and larger where the measure $\mu$ is sparse (to the extent this is possible).  Therefore, an investigation of the measures $\{|\varphi_n(z;\mu)|^2d\mu(z)\}_{n\in\bbN}$ can be thought of as an examination of how well the orthonormal polynomials ``smooth out" the measure $\mu$. The collection of weak limits of the measures $\{|\varphi_n(z;\mu)|^2d\mu(z)\}_{n\in\bbN}$ will be called the \textit{weak asymptotic measures} of the measure $\mu$.

There have been many results investigating both ratio asymptotics and the weak asymptotics of the measures $\{|\varphi_n(z;\mu)|^2d\mu(z)\}_{n\in\bbN}$ in a variety of settings (see \cite{SimaRat,SimonWeak,OPUC1,OPUC2}), and in every well-studied case it is known that the existence of the limit (\ref{ratdefn}) implies the existence of the weak limit
\begin{align}\label{wlimdefn}
\wlim_{\nri}|\varphi_n(z;\mu)|^2d\mu(z).
\end{align}
We will show that this holds true in a much more general context (see Theorem \ref{eqo} below).  Before we can establish this result, we must discuss another problem that we will investigate. 

Closely related to the polynomial $\Phi_n$ is a linear operator called the \textit{Bergman shift operator}.  To define this operator, let $\mcm:L^2(\mu)\rightarrow L^2(\mu)$ be the map given by $(\mcm f)(z)=zf(z)$.  Let $\mcp$ be the closure of the span of the polynomials inside $L^2(\mu)$.  It is easy to see that $\mcm$ maps $\mcp$ to itself and that the restricted map has matrix representation
\[
M=\begin{pmatrix}
M_{11} & M_{12} & M_{13} & M_{14} & \cdots\\
M_{21} & M_{22} & M_{23} & M_{24} & \cdots\\
0 & M_{32} & M_{33} & M_{34} & \cdots\\
0 & 0 & M_{43} & M_{44} & \cdots\\
0 & 0 & 0 & M_{54} & \cdots\\
\vdots & \vdots & \vdots & \vdots & \ddots
\end{pmatrix}
\]
in terms of the orthonormal basis given by the orthonormal polynomials (see \cite{Shift}).  Indeed, since $M_{jk}=\langle z\varphi_{k-1},\varphi_{j-1}\rangle$, it is easy to see that $M_{jk}=0$ if $j>k+1$.  Thus, the operator $\mcm$ determines a Hessenberg matrix $M$.  The relationship between $\Phi_n$ and the matrix $M$ is the following well-known result, which is contained in Proposition 2.2 in \cite{WeakCD}.

\begin{prop}\label{det}
Let $\pi_n$ be the projection onto the $n$-dimensional subspace given by the span of $\{1,z,\ldots,z^{n-1}\}$ inside $\mcp$.  The polynomial $\Phi_n(\mu)$ and the matrix $M$ are related by
\begin{align}\label{keydet}
\Phi_n(z;\mu)=\det\left(z-\pi_n M\pi_n\right).
\end{align}
\end{prop}

One is often interested in studying properties of the matrix $M$.  Of particular interest is the asymptotics of the diagonals of the matrix $M$; that is, one is interested in understanding
\begin{align}\label{intint}
\lim_{\nri}M_{n-j,n}=\lim_{\nri}\int_{\bbC}\overline{\varphi_{n-j-1}(z;\mu)}z\varphi_{n-1}(z;\mu)d\mu(z),\qquad j\in\bbZ
\end{align}
or even knowing if this limit exists.  Substantial results along this line of inquiry were obtained in \cite{Shift}.  Before we can explain these results, we have to define some additional notation.

Given a bounded region $G$ in the complex plane whose complement is simply connected in the extended complex plane, there is a unique conformal bijection $\phi_G$ that maps $\barc\setminus\barg$ to $\barc\setminus\bard$ satisfying $\phi_G(\infty)=\infty$ and $\phi_G'(\infty)>0$.  Let us denote the inverse to this map by $\psi_G$ and write the Laurent expansion of $\psi_G$ around infinity as
\begin{align}\label{psiexpd}
\psi_G(w)=b_{-1}w+b_0+\frac{b_1}{w}+\frac{b_2}{w^2}+\cdots,
\end{align}
where each $b_i\in\bbC$ and $b_{-1}>0$.  It is very often the case that if $\partial G\subseteq\supp(\mu)$, then the ratio $\varphi_n(z;\mu)\varphi_{n-1}(z;\mu)^{-1}$ approaches $\phi_G(z)$ when $|z|$ is sufficiently large (see \cite{SimBlob,Suetin} for examples).  Furthermore, it is often the case that the probability measures $\{|\varphi_n|^2d\mu\}_{n\in\bbN}$ converge weakly to the equilibrium measure of $\barg$ as $\nri$ (see \cite{SimBlob} for examples).  Therefore, in such cases one expects the integral in (\ref{intint}) to approach
\[
\int_{\partial G}z\phi_G(z)^jd\omega_{\textrm{eq},\barg}(z)=\int_0^{2\pi}\psi_G(\eitheta)e^{ij\theta}\frac{d\theta}{2\pi}=b_j,
\]
where $\omega_{\textrm{eq},\barg}$ is the equilibrium measure for $\barg$.  Indeed this was shown to be the case in the setting considered in \cite{Shift}.

This heuristic argument fails for several reasons.  Most notably, when one has ratio asymptotic results, it is usually the case that one obtains convergence outside the support of the measure (see for example \cite{SimaRat}).  However, for the above argument to work, one needs convergence inside the support of the measure.  Sometimes this convergence does hold (as in the case of area measure on the annulus $\{z:1/2<|z|<1\}$; see also \cite{MDPoly}), but quite often it is true that examining $\varphi_n\varphi_{n-1}^{-1}$ does not even make sense inside the support of the measure because $\varphi_{n-1}$ can vanish there.  However, it is shown in \cite{Shift} that if the region $G$ satisfies some mild smoothness conditions and if $\mu$ is area measure on $G$, then $M_{n-j,n}\rightarrow b_j$ as $\nri$ for every $j\in\bbN$.  The results of \cite{Shift} also include estimates on the rate of convergence.

Our goal here is to consider more general measures than those considered in \cite{Shift} and try to deduce comparable results.  As is often the case, this added level of generality will come at the expense of some precision in our results.  As an added bonus, our investigation of this problem yields a result concerning the moments of the weak asymptotic measures.  The main result is the following:

\begin{theorem}\label{eqo}
Let $\mu$ be a compactly supported and finite measure.  The following statements are equivalent:
\begin{enumerate}
\item  There exists a positive real number $R$ and a function $f(z)$ analytic in $\{z:R<|z|\leq\infty\}$ such that
\begin{align}\label{fdef}
\lim_{\nri}\frac{\Phi_n(z;\mu)}{\Phi_{n+1}(z;\mu)}=f(z),\qquad R<|z|\leq\infty.
\end{align}
\item  For every $j\in\bbN_0$, the following limit exists:
\begin{align}\label{kappa2}
\lim_{\nri}\frac{\kappa_{n-1-j}}{\kappa_{n-1}}M_{n-j,n}.
\end{align}
\end{enumerate}
If either of the above two conditions are satisfied, then for every $j\in\bbN_0$ the limit
\begin{align}\label{weak1}
\lim_{\nri}\int z^j|\varphi_n(z;\mu)|^2d\mu(z)
\end{align}
exists and is uniquely determined by the limits in $(\ref{fdef})$ or $(\ref{kappa2})$.
\end{theorem}

\noindent\textit{Remark.}  We will see later that the converse to the last part of Theorem \ref{eqo} is false by appealing to examples in cases where the measure $\mu$ is supported on the unit circle or the real line.

\vspace{2mm}

Theorem \ref{eqo} admits many applications and will in fact allow us to obtain new proofs of some previously known facts concerning ratio asymptotics and weak asymptotic measures.  Our main emphasis will be on new results, but we will also mention some of the known results that follow easily from Theorem \ref{eqo}.

Our first corollary of Theorem \ref{eqo} concerns ratio asymptotics of the orthonormal polynomials and the asymptotic structure of the matrix $M$ along its diagonals.  Let us denote by $\mcr$ the right shift operator on $\ell^2(\bbN)$ and by $\mcl$ the left shift operator on $\ell^2(\bbN)$.  Following the notation in \cite{Feintuch}, we will say that $M$ is \textit{weakly asymptotically Toeplitz} if the sequence of operators
\begin{align}\label{shifts}
\left\{\mcl^nM\mcr^n\right\}_{n\geq0}
\end{align}
converges weakly to a Toeplitz operator $T$ as $\nri$.

With the above terminology, we can state our first corollary.

\begin{corollary}\label{matif}
Let $\mu$ be a compactly supported and finite measure satisfying 
\begin{align}\label{posinf}
\liminf_{\nri}\kappa_n(\mu)\kappa_{n+1}(\mu)^{-1}>0.
\end{align}
The following statements are equivalent:
\begin{enumerate}
\item  There exists a positive real number $R$ and a function $f(z)$ analytic in $\{z:R<|z|\leq\infty\}$ such that
\[
\lim_{\nri}\frac{\varphi_n(z;\mu)}{\varphi_{n+1}(z;\mu)}=f(z),\qquad R<|z|\leq\infty.
\]
\item  The matrix $M$ is weakly asymptotically Toeplitz.
\end{enumerate}
\end{corollary}

\noindent\textit{Remark.}  We will have more to say later about the situation when the limit in (\ref{posinf}) is equal to zero.

\begin{proof}
Immediate from Theorem \ref{eqo}.
\end{proof}

Much more is known about orthogonal polynomials on the real line or the unit circle than in more general settings.  This is due in large part to the existence of a finite term recursion relation that exists for the monic orthogonal and orthonormal polynomials in these settings.  These relations give us explicit formulae for the entries of the matrix $M$.  The next two corollaries, which follow easily from Theorem \ref{eqo}, are known results in these more classical settings.  The first of these corollaries is a 2004 result originally due to Simon.  The second is a 2001 result due to Khruschev.

\begin{corollary}[Simon \cite{SimonWeak}]\label{oprlrat}
Let $\mu$ be be a finite measure supported on an infinite yet compact subset of the real line.  The monic orthogonal polynomials for the measure $\mu$ exhibit ratio asymptotics if and only if $\lim_{\nri}\kappa_n\kappa_{n-1}^{-1}$ exists and $\lim_{\nri}M_{n,n}$ exists.
\end{corollary}

\begin{corollary}[Khruschev \cite{Kr}]\label{opucweak}
Let $\mu$ be a probability measure supported on an infinite subset of the unit circle.  If for all $k\in\bbN$ it holds that $\lim_{\nri}\Phi_n(0;\mu)\Phi_{n+k}(0;\mu)=0$, then
\[
\wlim_{\nri}|\varphi_n(z;\mu)|^2d\mu(z)=\frac{d|z|}{2\pi}.
\]
\end{corollary}

Theorem \ref{eqo} and Corollary \ref{matif} naturally lead us to the following question: if for a given measure $\mu$ the corresponding orthonormal polynomials exhibit ratio asymptotics, what  does the limiting Toeplitz matrix of Corollary \ref{matif} look like and what are the limits of the moments in (\ref{weak1})?  This question motivates many of the results that follow.

We begin by relating the limiting function appearing in the ratio asymptotics to the limiting Toeplitz matrix in Corollary \ref{matif}.  We will continue to consider the case in which $\liminf_{\nri}\kappa_n\kappa_{n+1}^{-1}>0$.  In this case, the limiting function $f$ in Corollary \ref{matif} must have a Laurent coefficient of $z^{-1}$ that is nonzero.  Therefore, $f$ is injective on the set $\{z:|z|>R\}$ for some $R>0$.  It follows that there is a function $g(z)$ satisfying
\[
\left(\frac{1}{f(g(z))}\right)=z=g\left(\frac{1}{f(z)}\right),\qquad|z|>R.
\]
Let us write
\begin{align}\label{gdef}
g(z)=\beta_{-1}z+\beta_0+\frac{\beta_1}{z}+\frac{\beta_2}{z^2}+\cdots
\end{align}
With this notation, we can state our next result.

\begin{theorem}\label{whatlimit}
Let $\mu$ be a finite measure with compact support so that condition (a) in Corollary \ref{matif} is satisfied and (\ref{posinf}) holds.  With $g$ defined as in (\ref{gdef}), we have
\[
\lim_{\nri}M_{n-j,n}=\beta_j
\]
for all $j\geq-1$.
\end{theorem}

Theorem \ref{whatlimit} generalizes some previously known results.  For example, the results in \cite{Shift} established this fact when $\mu$ is area measure on a sufficiently smooth Jordan region.  Additionally, the main result in \cite{EGST} obtains a similar conclusion assuming $\mu$ is supported on a finite union of Jordan arcs, $\barc\setminus\supp(\mu)$ is simply connected, and some additional technical hypotheses.


One immediate consequence of Theorem \ref{whatlimit} is Corollary \ref{blobhess}.  If $T$ is a Toeplitz matrix and $T_{i,j}=a_{j-i}$ then we say that the \textit{symbol} of $T$ is the function $F$ where
\[
F(z)=\sum_{j=-\infty}^{\infty}a_jz^{-j}.
\]
The following result is a direct consequence of Theorem \ref{whatlimit} and the results in \cite{SimBlob}.

\begin{corollary}\label{blobhess}
Let $G$ be a Jordan domain with analytic boundary and suppose that $\rho<1$ is fixed so that $\{z:|z|>\rho\}$ is in the domain of $\psi_G$.  Consider the measure $\mutil(r\eitheta)=h(r\eitheta)\left(\nu(\theta)\otimes\tau(r)\right)+\sigma_2(r\eitheta)$ where
\begin{enumerate}
\item $\tau$ is a finite measure on $[\rho,1]$ and $1\in\supp(\tau)$,
\item $h(z)$ is a continuous function on $\bard$ that is non-vanishing in a neighborhood of $\partial\bbD$,
\item $\sigma_2$ is a finite measure carried by $\{z:\rho<|z|\leq1\}$ that satisfies
\[
\lim_{t\rightarrow\infty}\frac{\int|z|^td\sigma_2(z)}{\int|z|^td\tau(z)}=0,
\]
\item $\nu$ is a finite Szeg\H{o} measure on the unit circle.
\end{enumerate}
Let $\mu$ be the measure on $\bbC$ be given by
\[
\mu=(\psi_G)_*\mutil+\sigma_1+\sum_{j=1}^m\alpha_j\delta_{z_j}+\sum_{j=1}^{\ell}\beta_j\delta_{\zeta_j}
\]
where $\supp(\sigma_1)\subseteq G$, $\alpha_j,\beta_j>0$, $z_j\not\in\barg$ for all $j\in\{1,\ldots,m\}$, and $\zeta_j\in\partial G$ for all $j\in\{1,\ldots,\ell\}$.  Then the matrix $M$ is asymptotically Toeplitz and the limiting matrix has symbol $\psi_G(z)$.
\end{corollary}

\begin{proof}
The results of \cite{SimBlob} imply that if $\mu$ is of the desired form, then when $|z|$ is sufficiently large it holds that
\[
\lim_{\nri}\frac{\varphi_n(z;\mu)}{\varphi_{n+1}(z;\mu)}=\frac{1}{\phi_G(z)},
\]
so result now follows from Theorem \ref{whatlimit}.
\end{proof}

To address the second portion of the aforementioned question, let us focus on the relationship between ratio asymptotics of the monic orthogonal polynomials and the limits of the moments in (\ref{weak1}).  Theorem \ref{eqo} tells us that the weak asymptotic measures all have the same moments whenever the monic orthogonal polynomials exhibit ratio asymptotics.  The next natural step is to calculate these moments.  As Theorem \ref{eqo} suggests, calculating the moments of the weak asymptotic measures is possible when one knows the limiting function of the ratio asymptotics.  Our next result makes this more precise.

\begin{theorem}\label{equib}
Let $\mu$ be a compactly supported and finite measure.
Let $K$ be a compact set bounded by a Jordan curve and let $\phi_{K}$ denote the conformal bijection sending $\barc\setminus K$ to $\barc\setminus\bard$ satisfying $\phi_{K}(\infty)=\infty$ and $\phi_{K}'(\infty)>0$.  If
\begin{align}\label{phigrat}
\lim_{\nri}\frac{\varphi_n(z;\mu)}{\varphi_{n+1}(z;\mu)}=\frac{1}{\phi_{K}(z)},
\end{align}
when $|z|$ is sufficiently large and $\psi_K$ is the inverse to $\phi_K$, then for every $j\in\bbN$ it holds that
\[
\lim_{\nri}\int_{\bbC}z^j|\varphi_n(z;\mu)|^2d\mu(z)=\int_{0}^{2\pi}\psi_K(\eitheta)^j\frac{d\theta}{2\pi}.
\]
\end{theorem}

We will prove Theorem \ref{equib} in Section \ref{jordan}.  The next two corollaries are applications of Theorem \ref{equib} that we will prove in Sections \ref{jordan} and \ref{mass} respectively.

\begin{corollary}\label{squareweak}
If $\mu$ is area measure on a bounded and simply-connected region $G$ whose boundary is piecewise analytic without cusps, then
\[
\wlim_{\nri}|\varphi_n(z;\mu)|^2d\mu(z)=d\omega_{\textrm{eq},\barg}(z).
\]
\end{corollary}

\begin{corollary}\label{nvyweak}
Let $\mu$ be a measure that can be written as
\[
\mu=\mu_1+\mu_2+\mu_3
\]
where $\mu_1$ satisfies $\mu_1(\bbD)=\mu_1(\bbC)$, $\mu_2$ is a measure on the unit circle of the form $w(\theta)d\theta/2\pi$ where
\[
\int_0^{2\pi}\log(w(\theta))d\theta>-\infty,
\]
and $\mu_3$ is a purely discrete measure supported on $\bbC\setminus\bard$ whose mass points $\{z_j\}_{j\in\bbN}$ satisfy the balschke condition:
\[
\sum_{j=1}^{\infty}|z_j|-1<\infty.
\]
 Then
\[
\wlim_{\nri}|\varphi_n(z;\mu)|^2d\mu(z)=d\omega_{\bard}(z),
\]
where $\omega_{\bard}$ is the equilibrium measure for the unit disk.
\end{corollary}

Theorem \ref{eqo} tells us that if the monic orthogonal polynomials exhibit ratio asymptotics when $|z|$ is sufficiently large, then the weak asymptotic measures all have the same moments.  There is another canonical measure associated to the orthonormal polynomials, namely the normalized zero counting measure.  We define the measure $\nu_n$ by
\begin{align}\label{nudef}
d\nu_n=\frac{1}{n}\sum_{j=1}^n\delta_{z_j},
\end{align}
where $\{z_1,\ldots,z_n\}$ is the set of zeros of $\varphi_n(z;\mu)$, each listed a number of times equal to its multiplicity as a zero of $\varphi_n(\cdot;\mu)$.  With this notation, we can state our next result.

\begin{corollary}\label{weakzero}
Let $\mu$ be as in Theorem \ref{eqo}.  If there exists a positive real number $R$ and a function $f(z)$ analytic in $\{z:R<|z|\leq\infty\}$ such that
\[
\lim_{\nri}\frac{\Phi_n(z;\mu)}{\Phi_{n+1}(z;\mu)}=f(z),\qquad R<|z|\leq\infty,
\]
then for every $j\in\bbN$, the following limit exists:
\begin{align}\label{weak2}
\lim_{\nri}\int z^jd\nu_n(z),
\end{align}
and it is equal to the limit in (\ref{weak1}).
\end{corollary}

\begin{proof}
Let us recall \cite[Proposition 2.3]{WeakCD}, which tells us that
\[
\left|\int z^jd\nu_{n}(z)-\frac{1}{n}\sum_{k=0}^{n-1}\int z^j|\varphi_k(z;\mu)|^2d\mu(z)\right|\leq\frac{2j\| M\|}{n}.
\]
The result now follows from Theorem \ref{eqo}.
\end{proof}

\vspace{2mm}



The next Section is devoted to the proofs of the main results we have just discussed.  The remaining sections are devoted to the proofs of the corollaries and the investigation of examples and special cases to which we can apply our results.

\vspace{2mm}

\noindent\textbf{Acknowledgements.}  It is a pleasure to thank Ed Saff for encouraging me to pursue this line of inquiry and for much useful discussion and feedback concerning this work.  I would also like to thank Barry Simon for useful discussion.

\section{Proof of the Main Theorems}\label{equiv}

Our goal in this section is to prove many of the main results stated in the previous section.  The key fact that we will use is Proposition \ref{det}, which will allow us to write an explicit Laurent series for the ratio of consecutive monic orthogonal polynomials when $|z|$ is large.  We begin with the following lemma:

\begin{lemma}\label{repeat}
Let $\{i_1,\ldots,i_k\}\in\bbN^k$ satisfy
\[
i_{j+1}\geq i_j-1,\qquad 1\leq j\leq k-1.
\]
Let $i_m$ and $i_{m'}$ be two indices satisfying the following conditions:
\begin{enumerate}
\item $m<m'$,
\item $i_j\neq i_{\ell}$ for $j,\ell\in\{m,\ldots,m'-1\}$ and $j\neq \ell$,
\item $i_m=i_{m'}$.
\end{enumerate}
Let us denote
\[
i_*=\min\{i_m,i_{m+1},\ldots,i_{m'}\},\qquad
i^*=\max\{i_m,i_{m+1},\ldots,i_{m'}\}.
\]
Then
\[
M_{i_m,i_{m+1}}M_{i_{m+1},i_{m+2}}\cdots M_{i_{m'-1},i_{m'}}=\frac{\kappa_{i_*-1}}{\kappa_{i^*-1}}M_{i_*,i^*}.
\]
\end{lemma}

\begin{proof}
Since $i_{j+1}\geq i_j-1$ when $1\leq j\leq k-1$, the only way the hypotheses of the lemma can be satisfied is if
\begin{align*}
&M_{i_m,i_{m+1}}M_{i_{m+1},i_{m+2}}\cdots M_{i_{m'}-1,i_{m'}}=\\
&\qquad\qquad\qquad M_{i_m,i_{m}-1}M_{i_{m}-1,i_{m}-2}\cdots M_{i_*+1,i_*}M_{i_*,i^*}M_{i^*,i^*-1}\cdots M_{i_{m'}+1,i_{m'}}.
\end{align*}
It is straightforward to verify that $M_{n,n-1}=\kappa_{n-2}\kappa_{n-1}^{-1}$, and so the conclusion of the lemma follows.
\end{proof}

\begin{proof}[Proof of Theorem \ref{eqo}]
Let us assume that there is an analytic function $f(z)$ and a real number $R$ so that
\[
\lim_{\nri}\frac{\Phi_n(z;\mu)}{\Phi_{n+1}(z;\mu)}=f(z),\qquad R<|z|\leq\infty.
\]
Let $\rho_{M,n}(z)=(z-\pi_nM\pi_n)^{-1}$ be the resolvent of the truncated $M$ matrix.  By (\ref{keydet}) and Cramer's rule, we can write (when $|z|$ is sufficiently large)
\begin{align}\label{cram}
\frac{\Phi_{n-1}(z;\mu)}{\Phi_{n}(z;\mu)}=\rho_{M,n}(z)_{n,n}=\frac{1}{z}\sum_{j=0}^{\infty}\frac{\left((\pi_nM\pi_n)^j\right)_{n,n}}{z^j}
\end{align}
(see also \cite[equation (2.21)]{SimonWeak}).  Therefore, we see that
\[
\lim_{\nri}\left((\pi_nM\pi_n)^j\right)_{n,n}
\]
converges to the $(j+1)^{st}$ coefficient in the Laurent expansion of the function $f(z)$ around infinity.  We conclude that
\[
\lim_{\nri}\left((\pi_nM\pi_n)^j\right)_{n,n}
\]
exists for every $j\in\bbN$.  The proof proceeds now by induction.
For the base case of the induction, we set $j=0$ in (\ref{kappa2}), and note that
\[
\lim_{\nri}\frac{\kappa_{n-1}}{\kappa_{n-1}}M_{n,n}=\lim_{\nri}(\pi_nM\pi_n)_{n,n},
\]
and we have just seen that this limit exists.

To state our induction hypothesis, we need some auxiliary notation.  Let $[k]=\{0,1,\ldots,k\}$ and for each $k\in\bbN_0$ define $L(k)$ by
\begin{align}\label{ldef}
L(k)=\left\{\{i_0,i_1,\ldots,i_{k+1}\}\in[k]^{k+2}:i_{j+1}\leq i_j+1\mbox{ if } 0\leq j\leq k;\,i_{k+1}=i_0=0\right\}.
\end{align}
We know that
\begin{align}\label{longsum}
\left((\pi_nM\pi_n)^{k+1}\right)_{n,n}=\sum_{i_1,\ldots,i_{k}=0}^{n-1}M_{n,n-i_i}M_{n-i_1,n-i_2}\cdots M_{n-i_{k},n}.
\end{align}
Notice that in the sum (\ref{longsum}), the Hessenberg structure of the matrix $M$ requires $n-i_{j+1}\geq n-i_j-1$ for $j=1,\ldots,k-1$ and $n-i_1\geq n-1$ in order for the corresponding term in the sum to be non-zero.  In other words, to every non-zero term in the sum (\ref{longsum}), there corresponds a unique element of $L(k)$.  Therefore, we may rewrite (\ref{longsum}) as
\begin{align}\label{longsum2}
\left((\pi_nM\pi_n)^{k+1}\right)_{n,n}=\sum_{\{i_0,\ldots,i_{k+1}\}\in L(k)}M_{n-i_0,n-i_i}M_{n-i_1,n-i_2}\cdots M_{n-i_{k},n-i_{k+1}}.
\end{align}
Notice that $L(0)=\{\{0,0\}\}$ and we have already verified that the corresponding term in (\ref{longsum2}) approaches a limit as $\nri$.  For our induction hypothesis, we will assume that for every $0\leq t<k$ and every $\{i_0,\ldots,i_{t+1}\}\in L(t)$, the following limit exists:
\begin{align}\label{tterm}
\lim_{\nri}M_{n-i_0,n-i_1}M_{n-i_1,n-i_2}\cdots \cdots M_{n-i_t,n-i_{t+1}}.
\end{align}
In particular, notice that $\{0,1,2,\ldots,j,0\}\in L(j)$ and so our induction hypothesis implies
\[
\lim_{\nri}\frac{\kappa_{n-1-j}}{\kappa_{n-1}}M_{n-j,n}
\]
exists for all $j<k$.

For every $x\in L(k)$, let us break the element $x$ into subchains
\[
x=\{0,i_1,\ldots,i_{m_1}\}\cup\{i_{m_1+1},\ldots,i_{m_2}\}\cup\cdots\cup\{i_{m_{p-1}+1},\ldots,i_k,0\},
\]
where 
\begin{align*}
&n>n-i_1>\cdots >n-i_{m_1}\leq n-i_{m_1+1}\\
&n-i_{m_1+1}>n-i_{m_1+2}>\cdots>n-i_{m_2}\leq n-i_{m_2+1}\\
&\vdots\\
&n-i_{m_{p-1}+1}>n-i_{m_{p-1}+2}>\cdots>n-i_k.
\end{align*}
There are now two cases to consider.

\vspace{3mm}

\noindent\underline{Case 1:}  $x$ is exactly one subchain 

\vspace{3mm}

In this case, $i_{\ell}=\ell$ for all $\ell\in\{0,1,\ldots,k\}$ so Lemma \ref{repeat} implies the term in (\ref{longsum2}) corresponding to $x$ is exactly $\kappa_{n-k-1}\kappa_{n-1}^{-1}M_{n-k,n}$.

\vspace{3mm}

\noindent\underline{Case 2:}  $x$ splits into more than one subchain 

\vspace{3mm}

In this case, consider the term in (\ref{longsum2}) corresponding to $x$:
\begin{align}\label{xterm}
M_{n,n-i_1}M_{n-i_1,n-i_2}\cdots \cdots M_{n-i_k,n}.
\end{align}
Since $x$ splits into more than one subchain, there must be some $\ell\in\{0,1,\ldots,m_1\}$ so that $i_{m_1+1}=i_{\ell}=\ell$.  Therefore, somewhere in (\ref{xterm}) lies the expression
\begin{align}\label{jfac}
M_{n-\ell,n-\ell-1}M_{n-\ell-1,n-\ell-2}\cdots M_{n-\ell-q,n-\ell}
=\frac{\kappa_{n-\ell-q-1}}{\kappa_{n-\ell-1}}M_{n-\ell-q,n-\ell}
\end{align}
for some natural number $q<k$ (we used Lemma \ref{repeat} here).  The induction hypothesis implies that if we take $\nri$ in (\ref{jfac}) then the expression approaches a limit.  If we remove the factor (\ref{jfac}) from (\ref{xterm}), then we are left with
\begin{align}\label{takeout}
M_{n,n-i_1}\cdots M_{n-i_{\ell-1},n-i_{\ell}}M_{n-i_{\ell},n-i_{m_1+2}}\cdots M_{n-i_k,n},
\end{align}
which corresponds to
\[
\{i_0,\ldots,i_{\ell},i_{m_1+2},\ldots,i_k,i_{k+1}\}\in L(k+\ell-m_1-1).
\] 
Therefore, we may apply the induction hypothesis to the factor (\ref{takeout}) to see that it also approaches a limit as $\nri$.  We conclude that the term corresponding to $x$ in (\ref{longsum2}) approaches a limit as $\nri$.

\vspace{2mm}

Case 2 establishes convergence of the term corresponding to each element of $L(k)$ except $\{0,1,2,\ldots,k,0\}$.  Put differently, we know that every term in (\ref{longsum2}) except $\kappa_{n-k-1}\kappa_{n-1}^{-1}M_{n-k,n}$ approaches a limit as $\nri$.  We know the expression in (\ref{longsum2}) approaches a limit as $\nri$, so we must have $\kappa_{n-1-k}\kappa_{n-1}^{-1}M_{n-k,n}$ also approaches a limit, which completes the induction.

\vspace{2mm}

For the converse statement, suppose $\kappa_{n-1-k}\kappa_{n-1}^{-1}M_{n-k,n}$ approaches a limit as $\nri$ for every $k\in \bbN_0$.  Notice that the above proof shows that every term in the sum (\ref{longsum2}) can be written as a product of factors of the form
\[
\frac{\kappa_{n-1-q-\ell}}{\kappa_{n-1-q}}M_{n-\ell-q,n-q}
\]
for appropriate natural numbers $q$ and $\ell$.  Since we assuming all these factors approach a limit as $\nri$,
we conclude that for every $j\in\bbN_0$, the following limit exists:
\begin{align}\label{jlim}
\lim_{\nri}\left((\pi_nM\pi_n)^j\right)_{n,n}.
\end{align}
Let us denote the limit in (\ref{jlim}) by $A_j$ and consider the function
\begin{align}\label{toz}
\sum_{j=0}^{\infty}\frac{(\pi_nM\pi_n)^j_{n,n}}{z^{j+1}}-\sum_{j=0}^{\infty}\frac{A_j}{z^{j+1}},\qquad|z|=R>2\|\mcm\|.
\end{align}
If we split the sum after the first $J$ terms ($J\in\bbN$ here is arbitrary), then the difference between the first $J$ terms converges to $0$ as $\nri$, while
\[
\left|\sum_{j=J}^{\infty}\frac{(\pi_nM\pi_n)^j_{n,n}}{z^{j+1}}-\sum_{j=J}^{\infty}\frac{A_j}{z^{j+1}},\right|\leq
\sum_{j=J}^{\infty}\frac{2\|\mcm\|^j}{(2\|\mcm\|)^{j+1}}=\frac{2^{1-J}}{\|\mcm\|}.
\]
It follows that
\[
\lim_{\nri}\frac{\Phi_{n-1}(z;\mu)}{\Phi_n(z;\mu)}=\sum_{j=0}^{\infty}\frac{A_j}{z^j},
\]
on the circle $|z|=R$.  The maximum principle provides convergence outside this circle.

\vspace{2mm}

For the proof of the integral convergence statement, we must prove convergence of $\left(M^j\right)_{n,n}$ as $\nri$.  The idea behind the proof is the same as that given above.  We notice that
\begin{align}\label{jsum}
\left(M^{j}\right)_{n,n}=\sum_{i_1,\ldots,i_{j-1}=1}^{\infty}M_{n,i_i}M_{i_1,i_2}\cdots M_{i_{j-1},n}.
\end{align}
As before, there are only finitely many non-zero terms in this sum, each corresponding to a $(j+1)$-tuple $\{n,i_1,\ldots,i_{j-1},n\}$.  It is easy to see that we can always find a pair of indices $(i_m,i_{m'})$ in this $(j+1)$-tuple to which we can apply Lemma \ref{repeat}.  Furthermore, after removing $\{i_{m+1},\ldots,i_{m'}\}$ from this $(j+1)$-tuple, we are left with a $(j+1-m'+m)$-tuple that corresponds to a non-zero term in the sum (\ref{jsum}) with $j$ replaced by $j-m'+m$.  Therefore, we may again apply Lemma \ref{repeat}.  By iterating this procedure, we see that each non-zero term in (\ref{jsum}) can be written as a product of factors of the form
\[
\frac{\kappa_{n-1+q-\ell}}{\kappa_{n-1+q}}M_{n+q-\ell,n+q}
\]
for an appropriate choice of $q,\ell\in\{-j,-j+1,\ldots,j-1,j\}$.  Since we are assuming that these factors approach a limit as $\nri$, we obtain the desired convergence and the limit is determined by the limits in (\ref{kappa2}).  
\end{proof}

Now we will prove Theorem \ref{whatlimit}, which gives us precise information about the limiting Toeplitz matrix in the case $\liminf\kappa_{n}\kappa_{n+1}^{-1}>0$.

\begin{proof}[Proof of Theorem \ref{whatlimit}]
Let us write
\[
f(z)=\sum_{j=1}^{\infty}f_jz^{-j},\qquad R<|z|\leq\infty,\qquad f_1=\lim_{\nri}\frac{\kappa_n}{\kappa_{n+1}}.
\]
One can rewrite (\ref{gdef}) as
\begin{align}\label{glong}
\beta_{-1}+\frac{\beta_0}{z}+\frac{\beta_1}{z^2}+\frac{\beta_2}{z^3}+\cdots=\frac{g(z)}{z}=\sum_{j=1}^{\infty}f_jg(z)^{1-j}.
\end{align}
This easily implies $\beta_{-1}=f_1$.  If we define $h_k$ as a map on functions by $h_k(\ell)=g(z)\ell(z)-f_k$, then it follows from (\ref{glong}) that
\begin{align}\label{hsubdef}
\frac{f_{k+1}}{f_1}=\lim_{z\rightarrow\infty}\frac{g(z)}{\beta_{-1}}\left((h_k\circ h_{k-1}\circ\cdots\circ h_1)(z^{-1})\right).
\end{align}
One can check by hand that
\begin{align}\label{hone}
h_1(z^{-1})=\frac{\beta_0}{z}+\frac{\beta_1}{z^{2}}+\cdots
\end{align}
More generally, if one writes
\[
(h_k\circ h_{k-1}\circ\cdots\circ h_1)(z^{-1})=\sum_{j=1}^{\infty}\frac{c_j^{(k)}}{z^j},
\]
then we see from the definition of $h_k$ that
\begin{align}\label{crecur}
c_j^{(k+1)}=\sum_{m=-1}^{j-1}\beta_mc_{j-m}^{(k)}.
\end{align}

Let $B$ be the infinite matrix indexed by the natural numbers satisfying $B_{j,k}=\beta_{j-k}$ (so $B$ is the transpose of a Hessenberg matrix).  Let $c^{(k)}$ denote the infinite column vector whose entry in the $j^{th}$ row is $c_{j}^{(k)}$.  With this notation, the recursion (\ref{crecur}) can be rewritten as
\[
c^{(k+1)}=Bc^{(k)}.
\]
Iterating this formula and taking (\ref{hone}) into account, one obtains
\[
\frac{f_{k+1}}{f_1}=c_1^{(k)}=\left(B^k\right)_{1,1}.
\]
In other words,
\begin{align}\label{indbase}
\lim_{\nri}\left((\pi_nM\pi_n)^k\right)_{n,n}=\left(B^k\right)_{1,1},\qquad k\in\bbN.
\end{align}
We have already seen that $M_{n-j,n}\rightarrow\beta_{j}$ when $j=-1$.  The proof for the remaining values of $j$ proceeds by induction on $j$, the base case $j=0$ being obvious from (\ref{indbase}) with $k=1$.

For our induction hypothesis, assume $\lim_{\nri}M_{n-j,n}=\beta_j$ for all $j<k$.  Then (\ref{indbase}) (with $k$ replaced by $k+1$) implies
\begin{align*}
\lim_{\nri}\sum_{i_1,\ldots,i_k=0}^{k}M_{n,n-i_1}M_{n-i_1,n-i_2}\cdots M_{n-i_k,n}
&=\sum_{i_1,\ldots,i_k=0}^{k}B_{1,1+i_1}B_{1+i_1,1+i_2}\cdots B_{1+i_k,1}\\
&=\sum_{i_1,\ldots,i_k=0}^{k}\beta_{-i_1}\beta_{i_1-i_2}\cdots\beta_{i_k}.
\end{align*}
Our induction hypotheses tells us that (with $i_0=i_{k+1}=0$)
\[
\lim_{\nri}M_{n-i_j,n-i_{j+1}}=\beta_{i_{j}-i_{j+1}},\qquad j\in\{0,1,\ldots,k\},\quad i_{j}-i_{j+1}<k.
\]
Canceling these terms in the above expression yields
\[
\lim_{\nri}M_{n,n-1}M_{n-1,n-2}\cdots M_{n-k,n}=\beta_{-1}^{k}\beta_k,
\]
from which it follows that $M_{n-k,n}\rightarrow\beta_k$ as $\nri$ as desired.
\end{proof}

We also state here the following result, which is a consequence of the proof of Theorem \ref{eqo}

\begin{prop}\label{normf}
Let $\mu$ be any measure of compact and infinite support in the complex plane.  Then
\[
\left\{\frac{\Phi_{n-1}(z;\mu)}{\Phi_{n}(z;\mu)}\right\}_{n\in\bbN}
\]
is a normal family on the set $\{z:|z|>\|\mcm\|\}$.
\end{prop}

\begin{proof}
We have seen in the proof of Theorem \ref{eqo} that if $|z|>\|\mcm\|$, then
\[
\frac{\Phi_{n-1}(z;\mu)}{\Phi_{n}(z;\mu)}=\sum_{j=0}^{\infty}\frac{((\pi_{n}M\pi_{n})^j)_{n,n}}{z^{j+1}}
\]
so that
\[
\left|\frac{\Phi_{n-1}(z;\mu)}{\Phi_{n}(z;\mu)}\right|\leq\sum_{j=0}^{\infty}\frac{|((\pi_{n}M\pi_{n})^j)_{n,n}|}{|z|^{j+1}}\leq\frac{1}{|z|}\sum_{j=0}^{\infty}\frac{\|\mcm\|^j}{|z|^{j}}.
\]
The desired conclusion now follows easily from Montel's Theorem.
\end{proof}

Proposition \ref{normf} tells us that we can always find a subsequence through which one observes ratio asymptotics of the monic orthogonal polynomials.  In Section \ref{ratsubs}, we will determine precisely what functions $f$ can occur as the limit of $\Phi_{n-1}\Phi_n^{-1}$ as $n$ tends to infinity through some subsequence and the measure of orthogonality has compact support in the real line.  We will also prove a similar result when the measure of orthogonality is supported on the unit circle.

\vspace{2mm}

The remaining sections are devoted to applications of the results proven in this section.  In particular, we will prove all of the corollaries stated in Section \ref{intro} and look at some examples to which we can apply our new results.


\section{Measures on the Unit Circle and Real Line.}\label{opuc}

By appealing to examples of measures $\mu$ supported on the unit circle and real line, we will explore some further aspects of Theorem \ref{eqo}.  In particular, we will show that the converse to Theorem \ref{eqo} is false in that the convergence of the integrals in (\ref{weak1}) does not imply the measure $\mu$ admits ratio asymptotics for the monic orthogonal polynomials.  We will also explore consequences of Proposition \ref{normf} in these classical settings.

\subsection{Converse to Theorem \ref{eqo}}\label{con}  We begin by considering the converse to Theorem \ref{eqo}.  The simplest setting in which to consider properties of the matrix $M$ is when $\supp(\mu)\subseteq\bbR$.  When this occurs, the matrix $M$ is identically zero away from the three main diagonals, the matrix is self-adjoint, the diagonal entries are real, and the off-diagonal entries are positive.  Also, since we are assuming $\supp(\mu)$ is compact, then the conclusion (\ref{weak1}) is sufficient to guarantee the weak convergence of the measures $\{|\varphi_n|^2d\mu\}_{n\in\bbN}$ by the Stone-Weierstrass Theorem.

Both ratio asymptotics and the weak convergence properties of the measures $\{|\varphi_n|^2d\mu\}_{n\in\bbN}$ have been studied in \cite{SimonWeak}.  It is shown in \cite[Theorems 1 $\&$ 2]{SimonWeak} that the monic orthogonal polynomials exhibit ratio asymptotics if and only $M$ is weakly asymptotically Toeplitz.  It is also shown that the measures $|\varphi_n|^2d\mu$ converge weakly as $\nri$ if and only if the following limits exist:
\[
\lim_{\nri}M_{n,n},\qquad\lim_{\nri}M_{2n,2n-1},\qquad\lim_{\nri}M_{2n+1,2n}.
\]
Clearly this is a weaker condition than the one required for ratio asymptotics, so the existence of a unique weak asymptotic measure does not imply that either of the two conditions $(a)$ or $(b)$ of Theorem \ref{eqo} hold.  Similar results hold when the measure $\mu$ is supported on the unit circle (see \cite[Chapter 9]{OPUC2}).

We will now provide a new proof of \cite[Theorem 1]{SimonWeak}.

\begin{proof}[Proof of Corollary \ref{oprlrat}]
As mentioned above, if $\supp(\mu)\subseteq\bbR$, the matrix $M$ is zero away from the three main diagonals.  Therefore, Theorem \ref{eqo} implies the monic orthogonal polynomials exhibit ratio asymptotics if and only if the limits
\[
\lim_{\nri}M_{n,n},\qquad\lim_{\nri}\frac{\kappa_{n-2}}{\kappa_{n-1}}M_{n-1,n}
\]
exist.  Since $M$ is self-adjoint, the second of these limits can be written as
\[
\lim_{\nri}\left(\frac{\kappa_{n-2}}{\kappa_{n-1}}\right)^2,
\]
which exists if and only if $\lim_{\nri}\kappa_{n-1}\kappa_{n}^{-1}$ exists.
\end{proof}

If $\mu$ is a probability measure and $\supp(\mu)\subseteq\{z:|z|=1\}$, then much structure of the matrix $M$ is known.  In this case, the matrix $M$ is sometimes called the GGT representation of the operator $\mcm$ (see \cite[Chapter 4]{OPUC1}).  It is well known that when $\supp(\mu)\subseteq\{z:|z|=1\}$ there is a sequence of complex numbers $\{\alpha_n\}_{n=0}^{\infty}$ in the unit disk (called \textit{Verblunsky Coefficients}) so that
\begin{align}\label{twoterm}
\Phi_{n+1}(z;\mu)=z\Phi_n(z;\mu)-\bar{\alpha}_nz^n\overline{\Phi_n(1/\bar{z};\mu)}.
\end{align}
If we set $\rho_n=\sqrt{1-|\alpha_n|^2}>0$, then \cite[Equation (4.1.6)]{OPUC1} tells us that
\[
M_{n-j,n}=
\begin{cases}
0, & \text{if }j<-1 \\
 \rho_{n-j-2}, & \text{if }j=-1\\
 -\bar{\alpha}_{n-1}\alpha_{n-j-2}\prod_{k=n-j-1}^{n-2}\rho_k, & \text{if }j\geq0
\end{cases}
\]
Also, it is well known that
\[
\kappa_n=\prod_{j=0}^{n-1}\rho_j^{-1}.
\]
Therefore, Theorem \ref{eqo} tells us that the polynomials $\{\Phi_n(\cdot;\mu)\}_{n\in\bbN}$ exhibit ratio asymptotics if and only if for every $j\in\bbN_0$
\begin{align}\label{ggtlim}
\lim_{\nri}-\bar{\alpha}_{n-1}\alpha_{n-j-2}\prod_{k=n-j-1}^{n-2}\rho^2_k
\end{align}
exists.  We also know from \cite[Chapter 9]{OPUC2} that the polynomials $\{\Phi_n(\cdot;\mu)\}_{n\in\bbN}$ exhibit ratio asymptotics if and only if one of the following two conditions hold:
\begin{enumerate}
\item  For every $k\in\bbN$ it holds that $\lim_{\nri}\alpha_n\alpha_{n+k}=0$.
\item There are complex numbers $a\in(0,1]$ and $\lambda\in\partial\bbD$ so that
\[
\lim_{\nri}|\alpha_n|=a,\qquad\lim_{\nri}\frac{\alpha_{n+1}}{\alpha_n}=\lambda.
\]
\end{enumerate}
It is not obvious that one of these two conditions being satisfied is equivalent to (\ref{ggtlim}), but Theorem \ref{eqo} tells us that it must be.  We leave the details to the reader but we will provide a new proof of a (well-known) consequence of $(a)$.

\begin{proof}[Proof of Corollary \ref{opucweak}]
Suppose that for all $k\in\bbN$ it holds that
\[
\lim_{\nri}\Phi_n(0;\mu)\Phi_{n+k}(0;\mu)=0.
\]
It follows easily from (\ref{twoterm}) that $-\bar{\alpha}_n=\Phi_n(0;\mu)$.  Therefore, we can rewrite our hypothesis as $\lim_{\nri}\alpha_n\alpha_{n+k}=0$.  It now follows that
\[
\lim_{\nri}-\bar{\alpha}_{n-1}\alpha_{n-k-2}\prod_{j=n-k-1}^{n-2}\rho^2_j=0,\qquad k\in\bbN.
\]
The proof of Theorem \ref{eqo} and (\ref{ggtlim}) allow us to deduce that
\[
\lim_{\nri}\int_{\partial\bbD}z^k|\varphi_n(z;\mu)|^2d\mu(z)=0,\qquad k\in\bbN,
\]
which is the desired conclusion.
\end{proof}

Notice that Corollary \ref{matif} requires the additional hypothesis $\liminf_{\nri}\kappa_n\kappa_{n+1}^{-1}>0$.  The following example shows that this is an essential hypothesis, for if it fails, the behavior of the ratio of consecutive orthonormal polynomials is trivial and tell us nothing.

\vspace{2mm}

\noindent\textbf{Example.}  Let $\mu$ be a probability measure (with infinite support) satisfying $\supp(\mu)\subseteq\partial\bbD$ and suppose the Verblunsky coefficients satisfy $\lim_{\nri}\alpha_n=1$.  In this case, $\kappa_n\kappa_{n+1}^{-1}\rightarrow0$.  Therefore,
\[
\lim_{\nri}\frac{\kappa_{n-1-j}}{\kappa_{n-1}}M_{n-j,n}=
\begin{cases}
0, & \text{if }j\neq0 \\
-1, & \text{if }j=0
\end{cases},
\]
and so we conclude from Theorem \ref{eqo} that there is an analytic function $f$ so that
\[
\lim_{\nri}\frac{\Phi_n(z;\mu)}{\Phi_{n+1}(z;\mu)}=f(z),\qquad|z|>1.
\]
Since $\kappa_n\kappa_{n+1}^{-1}\rightarrow0$, we can also write
\[
\lim_{\nri}\frac{\varphi_n(z;\mu)}{\varphi_{n+1}(z;\mu)}=0
\]
uniformly on compact subsets of $\barc\setminus\bard$.

Now consider the case when the recursion coefficients of $\mu$ satisfy
\[
\alpha_n=\left(1-\frac{1}{n+1}\right)e^{in^2}.
\]
In this case, $\lim_{\nri}M_{n,n}$ does not exist, so one does not have ratio asymptotics for the monic orthogonal polynomials.  However, Proposition \ref{normf} shows that $\Phi_n(\cdot;\mu)\Phi_{n+1}(\cdot;\mu)^{-1}$ is uniformly bounded on compact subsets of $\barc\setminus\bard$, so in this case we can still conclude
\[
\lim_{\nri}\frac{\varphi_n(z;\mu)}{\varphi_{n+1}(z;\mu)}=0
\]
uniformly on compact subsets of $\barc\setminus\bard$.  Therefore, if $\lim_{\nri}\kappa_n\kappa_{n+1}^{-1}=0$, ratio asymptotics for the orthonormal polynomials are not indicative of the behavior of the monic orthogonal polynomials.


\subsection{Ratio Asymptotics Along Subsequences}\label{ratsubs}  Now let us consider Proposition \ref{normf} in the context of orthogonal polynomials on the real line.  We have already seen exactly when the monic orthogonal polynomials exhibit ratio asymptotics as $\nri$ through $\bbN$, but Proposition \ref{normf} tells us that even if this does not occur, we can still find subsequences through which the monic orthogonal polynomials exhibit ratio asymptotics.  It turns out that when $\mu$ is supported on a compact subset of the real line, we can determine the possible set of limit points of $\Phi_{n-1}(z;\mu)\Phi_n(z;\mu)^{-1}$.  This is the content of our next result.

\begin{theorem}\label{oprlclass}
If $\mu$ is supported on a compact subset of the real line and $\mcn\subseteq\bbN$ is a subsequence so that
\begin{align}\label{nrat2}
\lim_{{\nri}\atop{n\in\mcn}}\frac{\Phi_{n-1}(z;\mu)}{\Phi_n(z;\mu)}=f(z),\qquad |z|>\|\mcm\|,
\end{align}
then there is a compactly supported probability measure $\nu$ on $\bbR$ so that
\begin{align}\label{fdef2}
f(z)=\int_{\bbR}\frac{1}{z-w}d\nu(w),\qquad |z|>\|\mcm\|.
\end{align}
Conversely, if $\nu$ and $f$ are related by (\ref{fdef2}), then there is a compactly supported probability measure $\mu$ on $\bbR$ and a subsequence $\mcn\subseteq\bbN$ so that (\ref{nrat2}) holds.
\end{theorem}

Before we can prove this result, we need to recall some terminology from spectral theory (many of the necessary facts can be found in \cite[Section 2.3]{Rice}).  We recall that a \textit{Herglotz function} is a function that is analytic in the upper half-plane and maps the upper half-plane to itself.  A \textit{discrete $m$-function} is defined to be a Herglotz function that can be analytically continued to $\bbC\setminus I$ for some bounded interval $I\subseteq\bbR$ and satisfies the following two properties:
\begin{enumerate}
\item[i)] the function is real valued on $\bbR\setminus I$,
\item[ii)] the function is $-z^{-1}+O(z^{-2})$ at infinity.
\end{enumerate}
The result we need is the following result known as the \textit{Herglotz Representation Theorem for Discrete $m$-functions} (see \cite[Theorem 2.3.6]{Rice}).

\begin{theorem}[\cite{Rice}]\label{HDMF}
A function $m$ that is analytic on $\{z\in\bbC:\mbox{Im}(z)>0\}$ is a discrete $m$-function if and only if there is a probability measure $\nu$ with support contained in a compact interval of the real line so that
\[
m(z)=\int\frac{d\nu(x)}{x-z}.
\]
\end{theorem}

\noindent Now we have what we need to prove Theorem \ref{oprlclass}.


\begin{proof}[Proof of Theorem \ref{oprlclass}]
Let $\mu$ be given and let $\mcn$ be chosen so that (\ref{nrat2}) holds along this subsequence.  We will rely heavily on the results stated in \cite[Proposition 2.3]{SimonWeak}.  First of all, we see that it was proven there that
\[
\left|\frac{\Phi_{n-1}(z;\mu)}{\Phi_n(z;\mu)}\right|\leq\frac{1}{|\mbox{Im}(z)|},
\]
so that the limit $f$ from (\ref{nrat2}) is analytic all the way up to the real line (not just on $\{z:|z|>\|\mcm\|\}$).  It was also proven there that if $\mbox{Im}(z)>0$ then
\[
\mbox{Im}\left(\frac{\Phi_{n-1}(z;\mu)}{\Phi_n(z;\mu)}\right)<0.
\]
It is obvious that $\Phi_{n-1}\Phi_n^{-1}$ is real valued on the real line away from the zeros of $\Phi_n$ so the same must be true of $f$ outside the convex hull of the support of $\mu$.  Combining these facts implies $-f$ is a discrete $m$-function, so the result follows from the representation theorem.

Conversely, suppose $\nu$ is a compactly supported probability measure on $\bbR$.  Let $M(\nu)$ be the corresponding matrix representing the Bergman Shift on $L^2(\nu)$.  If $\supp(\nu)$ is infinite, then let $\{a_n'\}_{n\geq1}$ and $\{b_n'\}_{n\in\bbN}$ be respectively the corresponding sequence of off-diagonal and diagonal elements of $M(\nu)$.  Construct two bounded sequences $\{a_n\}_{n\in\bbN}$ and $\{b_n\}_{n\in\bbN}$ that admit a subsequence $\mcn$ satisfying
\[
\lim_{{\nri}\atop{n\in\mcn}}b_{n-j}=b_{j+1}',\qquad\lim_{{\nri}\atop{n\in\mcn}}a_{n-j-1}=a_{j+1}',\qquad j\in\bbN_0.
\]
It is easy to see that such sequences $\{a_n(\mu)\}_{n\in\bbN}$ and $\{b_n(\mu)\}_{n\in\bbN}$ can be constructed.  If we let $\{b_n\}_{n\in\bbN}$ be the diagonal elements of a matrix $M$ and $\{a_n\}_{n\in\bbN}$ be the off-diagonal elements of the matrix $M$ then we may take $\mu$ to be the spectral measure of $M$ and the vector $e_1$.  Then $M=M(\mu)$ and the proof of Theorem \ref{eqo} implies that the monic orthogonal polynomials for $\mu$ exhibit ratio asymptotics through the subsequence $\mcn$ when $|z|$ is sufficiently large.  If we define $\mt_n$ by
\[
\mt_n=\begin{pmatrix}
b_n & a_{n-1} & 0 & \cdots\\
a_{n-1} & b_{n-1} & a_{n-2} & \cdots\\
0 & a_{n-2} & b_{n-2} & \ddots\\
\vdots & \vdots & \ddots & \ddots\\
0 & \cdots & a_1 & b_1
\end{pmatrix},
\]
then it is easy to see that $\mt_n$ converges strongly to $M(\nu)$ as $\nri$ through $\mcn$.  Therefore, the same is true of the resolvents when $|z|$ is sufficiently large.  Therefore, when $|z|$ is sufficiently large we can write
\[
f(z)=\lim_{{\nri}\atop{n\in\mcn}}\frac{\Phi_{n-1}(z;\mu)}{\Phi_n(z;\mu)}=\lim_{{\nri}\atop{n\in\mcn}}\rho_{M,n}(z)_{n,n}=\lim_{{\nri}\atop{n\in\mcn}}\rho_{\mt_n,n}(z)_{1,1}=\int_{\bbR}\frac{1}{z-w}d\nu(w)
\]
as desired.

If $\supp(\nu)$ is finite, then $\nu$ is the spectral measure of a finite matrix $M(\nu)$ with diagonal elements $\{b_n'\}_{n=1}^N$ and off-diagonal elements $\{a_n'\}_{n=1}^{N-1}$.  Choose $\mu$ so that for some subsequence $\mcn\subseteq\bbN$
\begin{align*}
\lim_{{\nri}\atop{n\in\mcn}}M_{n-j,n-j}(\mu)&=\begin{cases}
b_{N-j}' & \mbox{ if } 0\leq j\leq N-1\\
0 & \mbox{ otherwise}
\end{cases}\\
\lim_{{\nri}\atop{n\in\mcn}}M_{n-j,n-j-1}(\mu)&=\begin{cases}
a_{N-j-1}' & \mbox{ if } 0\leq j\leq N-2\\
0 & \mbox{ otherwise}
\end{cases}
\end{align*}
(again, it is easy to see that this can be done).  Then we may proceed as above to see that the limit function $f$ is of the desired form.
\end{proof}

\vspace{2mm}

\noindent\textbf{Example.}  In $\mu$ is supported on the real line and the corresponding sequences along the diagonals of $M$ satisfy
\[
\lim_{\nri}M_{n,n}=0,\qquad\lim_{\nri}M_{n,n-1}=1,
\]
then the corresponding measure $\nu$ from the proof of Theorem \ref{oprlclass} is
\[
d\nu(x)=\frac{1}{2\pi}\chi_{[-2,2]}(x)\sqrt{x^2-4}\,dx
\]
(see \cite[Equation 1.10.3]{Rice}).  Therefore, Theorem \ref{oprlclass} implies
\[
\lim_{\nri}\frac{\Phi_{n-1}(z;\mu)}{\Phi_n(z;\mu)}=\int_{-2}^2\frac{\sqrt{x^2-4}}{z-x}\frac{dx}{2\pi}.
\]
The results of \cite{SimonWeak} tell us that this is equal to $2(z+\sqrt{z^2-4})^{-1}$.  We leave the verification to the reader.

\vspace{2mm}

One can similarly investigate the functions that occur as the limiting function in (\ref{fdef}) when the measure $\mu$ is supported on the unit circle and $n$ tends to infinity through a subsequence.  Using the Herglotz Representation for Carath\'{e}odory Functions (see \cite[Theorem 2.3.5]{Rice}), one can show that every such limit can be expressed as the appropriate transform of a measure on the unit circle.  However, obtaining a converse statement as in Theorem \ref{oprlclass} is much more challenging.  The difficulty in adapting the proof of Theorem \ref{oprlclass} to the unit circle case is that when one ``flips" the $M$ matrix as in the proof of Theorem \ref{oprlclass}, the resulting matrix does not converge weakly to the $M$ matrix of a measure on the unit circle.  The subtlety is that $M_{1,1}=\bar{\alpha}_0$ (just one Verblunsky coefficient) while $M_{n,n}=-\bar{\alpha}_{n-1}\alpha_{n-2}$ (the product of two Verblunsky coefficients).  However, if we modify the ratios we are looking at, this problem can be avoided.

If $\mu$ is supported on the unit circle, let us define the degree $n$ \textit{paraorthogonal polynomial} $\Phi_n^{(-1)}(z;\mu)$ by
\[
\Phi_n^{(-1)}(z;\mu)=z\Phi_{n-1}(z;\mu)+z^{n-1}\overline{\Phi_{n-1}(1/\bar{z};\mu)}
\]
(see \cite[Section 2.2]{OPUC1}).  Notice that
\[
\frac{\Phi_n^{(-1)}(z;\mu)}{\Phi_{n-1}(z;\mu)}=z+\frac{z^{n-1}\overline{\Phi_{n-1}(1/\bar{z};\mu)}}{\Phi_{n-1}(z;\mu)},
\]
and so one easily sees that
\[
\left\{\frac{\Phi_{n-1}(z;\mu)}{\Phi_n^{(-1)}(z;\mu)}\right\}_{n\in\bbN}
\]
is a normal family on $\{z:|z|>1\}$.  Therefore, one can always extract a subsequence along which one observes uniform convergence on compact sets.

The recursion (\ref{twoterm}) makes it clear that if $M_{(-1)}$ is the $M$-matrix for $\mu$ with $\alpha_{n-1}$ replaced by $-1$ then
\[
\Phi_n^{(-1)}(z;\mu)=\det\left(z-\pi_nM_{(-1)}\pi_n\right).
\]
This change enables us to employ the trick used in the proof of Theorem \ref{oprlclass}.  The result is the following theorem.

\begin{theorem}\label{paraclass}
If $\mu$ is a probability measure with infinite support on $\partial\bbD$ and $\mcn\subseteq\bbN$ is a subsequence so that
\begin{align}\label{nrat3}
\lim_{{\nri}\atop{n\in\mcn}}\frac{\Phi_{n-1}(z;\mu)}{\Phi_n^{(-1)}(z;\mu)}=f(z)\qquad|z|>1.
\end{align}
Then there is a probability measure $\nu$ supported on $\partial\bbD$ so that
\begin{align}\label{everyf}
f(z)=\int_{\partial\bbD}\frac{1}{z-w}d\nu(w),\qquad|z|>1.
\end{align}
Conversely, if $f$ and $\nu$ are related by (\ref{everyf}) then there is a probability measure $\mu$ supported on $\partial\bbD$ and a subsequence $\mcn\subseteq\bbN$ so that (\ref{nrat3}) holds.
\end{theorem}

\begin{proof}
It is well-known that the matrix $\pi_nM_{(-1)}\pi_n$ is unitary with $n$ distinct eigenvalues.  If we define $\widetilde{M}_{(-1),n}$ by
\[
\left(\widetilde{M}_{(-1),n}\right)_{i,j}=\left(\pi_nM_{(-1)}\pi_n\right)_{n+1-j,n+1-i},\qquad i,j\in\{1,\ldots,n\},
\]
then
\[
\frac{\Phi_{n-1}(z;\mu)}{\Phi_n^{(-1)}(z;\mu)}=\left(\left(z-\pi_nM_{(-1)}\pi_n\right)^{-1}\right)_{n,n}=
\left(\left(z-\widetilde{M}_{(-1),n}\right)^{-1}\right)_{1,1}=\int_{\partial\bbD}\frac{1}{z-w}d\tilde{\nu}_n
\]
for some finitely supported probability measure $\tilde{\nu}_n$ on $\partial\bbD$.  It follows that if $\nu$ is any weak limit of the measures $\{\tilde{\nu}_n\}_{n\in\mcn}$ then
\[
\lim_{{\nri}\atop{n\in\mcn}}\frac{\Phi_{n-1}(z;\mu)}{\Phi_n^{(-1)}(z;\mu)}=f(z)=\int_{\partial\bbD}\frac{1}{z-w}d\nu(w),\qquad|z|>1,
\]
as desired.

The proof of the converse statement is very similar to the second part of the proof of Theorem \ref{oprlclass}, so we omit the details here.
\end{proof}

\noindent\textit{Remark.}  The proof of Theorem \ref{paraclass} can be adapted to produce a second proof of the first part of Theorem \ref{oprlclass}.

\vspace{2mm}

In the next section, we will leave the classical settings and discuss applications of our main results to more general measures.

\section{Measures on Jordan Regions.}\label{jordan}

We will use the term \textit{Jordan region} to refer to a region bounded by a Jordan curve.  We recall Carath\'{e}odory's Theorem (see Section 1.3 in \cite{GarnMar}), which tells us that any conformal bijection from such a region to the unit disk can be extended to a homeomorphism of the closure of the region with the closed unit disk. 

Let $\mu$ be a measure whose support is contained in the closure of some Jordan region $G$ and let $\sigma$ be any conformal bijection of $G$ with $\bbD$.  Let $\gamma_n$ be the balayage of the measure $|\varphi_n(z;\mu)|^2d\mu(z)$ onto $\partial G$.  If either of the conditions (a) or (b) of Theorem \ref{eqo} hold (in which case both hold), then we claim that the measures $\gamma_n$ converge weakly.

To see this, let $\gamma$ be any weak limit point of the measures $\{\gamma_n\}_{n\in\bbN}$.  Notice that Theorem \ref{eqo} tells us that all such weak limit points have the same moments; that is, they agree on the space of polynomials and hence they agree on the closure of this space in the $L^{\infty}(\partial G)$ norm.  By Mergelyan's Theorem (see \cite[Theorem 20.5]{Rudin}), the functions $\{\sigma^n\}_{n\in\bbN}$ are all in the closure of the space of polynomials in the $L^{\infty}(\partial G)$ norm.  Since
\[
\int_{\partial G}\overline{\sigma(z)^n}d\gamma(z)=\overline{\int_{\partial G}\sigma(z)^nd\gamma(z)},
\]
we conclude that all weak limit points agree with $\gamma$ on the closure of the set
\[
\mbox{span}\left\{\{\sigma(z)^n\}_{n\in\bbN_0}\,\bigcup\,\{\overline{\sigma(z)^n}\}_{n\in\bbN}\right\}
\]
in $L^{\infty}(\partial G)$.  By the complex Stone-Weierstrass Theorem, this is all continuous functions on $\partial G$, so $\gamma$ is the unique weak limit.

It is a separate matter to calculate the moments of the limiting measure.    Some results on this subject can be found in \cite[Chapter 9]{OPUC2} as well as \cite{SimBlob,SimaRat,SimonWeak}.  Our contribution to this effort is Theorem \ref{equib}, which we now prove.

\begin{proof}[Proof of Theorem \ref{equib}]
For every $n\in\bbN$, let us define the polynomial $H_n$ to be the polynomial part of the function $\psi_{K}^n$.  With this notation, we have
\[
\int_0^{2\pi}\psi_{K}(\eitheta)^j\frac{d\theta}{2\pi}
=\frac{1}{2\pi i}\int_{|z|=1}\frac{\psi_{K}(z)^j}{z}dz=H_j(0).
\]
Let us write
\begin{align*}
\psi_{K}(w)=d_{-1}w+d_0+\frac{d_1}{w}+\frac{d_2}{w^2}+\cdots,
\end{align*}
as the Laurent expansion of $\psi_{K}$ around infinity.  It is easy to see that if we set $d_{-k}=0$ for $k>1$, then
\begin{align}\label{hjo}
H_j(0)=\sum_{{\{i_1,\ldots,i_{j}\}\subseteq\bbZ^j}\atop{\sum i_k=0}}d_{i_1}d_{i_2}\cdots d_{i_j}.
\end{align}
Meanwhile,
\begin{align}\label{bigsum}
\lim_{\nri}\int_{\bbC}z^j|\varphi_n(z;\mu)|^2d\mu(z)=\lim_{\nri}\sum_{i_1,\ldots,i_{j-1}=n-j+1}^{n+j-1}M_{n,i_1}M_{i_1,i_2}\cdots M_{i_{j-1},n}.
\end{align}
The hypotheses imply that $\kappa_n\kappa_{n+1}^{-1}\rightarrow d_{-1}>0$ as $\nri$, so Theorem \ref{whatlimit} tells us that for every $j\in\bbN_0$
\[
\lim_{\nri}M_{n-j,n}=d_j.
\]

Now, to each term in the sum in (\ref{bigsum}) there corresponds a term in the sum (\ref{hjo}) by means of the correspondence
\[
M_{n,i_1}M_{i_1,i_2}\cdots M_{i_{j-1},n}\quad\longleftrightarrow\quad d_{i_1-n}d_{i_2-i_1}\cdots d_{n-i_{j-1}}.
\]
Therefore, every term in the sum (\ref{bigsum}) converges to a term in the sum (\ref{hjo}) as $\nri$ and every such non-trivial term is attained in this way, so we have
\[
\lim_{\nri}\int_{\bbC}z^j|\varphi_n(z;\mu)|^2d\mu(z)=\int_0^{2\pi}\psi_K(\eitheta)^j\frac{d\theta}{2\pi}.
\]
\end{proof}

If $G$ is a Jordan region then the equilibrium measure for $\barg$ is defined as the push-forward of the equilibrium measure of the unit disk under the map $\psi_G$.  Consequently, Theorem \ref{equib} and our discussion above immediately imply our next result.

\begin{corollary}\label{equibcor}
Let $\mu$ be supported on the closure of a Jordan region $G$.  If
\[
\lim_{\nri}\frac{\varphi_n(z;\mu)}{\varphi_{n+1}(z;\mu)}=\frac{1}{\phi_G(z)}
\]
when $|z|$ is sufficiently large and every weak asymptotic measure is supported on $\partial G$, then the equilibrium measure for $\barg$ is the unique weak asymptotic measure.
\end{corollary}

Now we can prove Corollary \ref{squareweak}.

\begin{proof}[Proof of Corollary \ref{squareweak}]
Let $\mu$ be area measure on a region $G$ that is bounded by a Jordan curve that is piecewise analytic without cusps.  We know from \cite[Theorem 1.2]{Corners} that
\[
\lim_{\nri}\frac{\varphi_n(z;\mu)}{\varphi_{n+1}(z;\mu)}=\frac{1}{\phi_G(z)}.
\]
Furthermore, it was shown in \cite[Lemma 7.6]{Corners} that if $K$ is a compact subset of $G$ then there is a constant $c_K$ so that
\[
\|\varphi_n(\cdot;\mu)\|_{L^{\infty}(K)}\leq c_Kn^{-1/2}.
\]
This implies
\[
\lim_{\nri}\int_K|\varphi_n(z;\mu)|^2d\mu(z)=0
\]
for every compact set $K\subset G$.  It follows easily that every weak asymptotic measure is supported on the boundary of $G$.  The desired conclusion now follows from Corollary \ref{equibcor}.
\end{proof}

Theorem \ref{equib} characterizes the moments of the weak asymptotic measures when the orthonormal polynomials exhibit a certain ratio asymptotic behavior. However, we have seen that in some cases, the liminf in (\ref{posinf}) converges to $0$, in which case Theorem \ref{equib} does not apply.  However, in this case we can still characterize the moments of the weak asymptotic measures.  Before we state our result, let us revisit the example from the end of Section \ref{opuc}.

\vspace{2mm}

\noindent\textbf{Example.}  Let $\mu$ be a probability measure satisfying $\supp(\mu)\subseteq\partial\bbD$ and suppose the recursion coefficients satisfy $\lim_{\nri}\alpha_n=1$.  In this case, \cite[Theorem 4.2.11]{OPUC1} tells us that the essential support of $\mu$ is $\{-1\}$.  We saw earlier that
\[
\lim_{\nri}\frac{\kappa_{n-1-j}}{\kappa_{n-1}}M_{n-j,n}=
\begin{cases}
0, & \text{if }j\neq0 \\
-1, & \text{if }j=0
\end{cases}.
\]
From this, it becomes clear that every term in the sum (\ref{bigsum}) converges to $0$ except the term corresponding to $i_1=i_2=\cdots=i_{j-1}=n$.  Therefore,
\begin{align*}
\lim_{\nri}\int_{\bbC}z^j|\varphi_n(z;\mu)|^2d\mu(z)
=\left(\lim_{\nri}M_{n,n}\right)^j
=(-1)^j.
\end{align*}
in accordance with \cite[Theorem 4.2.11]{OPUC1}.



\vspace{2mm}

The argument in the above example can be applied more generally and yields the following result, which is reminiscent of the equivalence of parts (ii) and (iii) of \cite[Theorem 4.2.11]{OPUC1}.

\begin{theorem}\label{equib2}
Let $\mu$ be a compactly supported and finite measure.
If $\lim_{\nri}\kappa_n\kappa_{n+1}^{-1}=0$ and there is a number $x$ so that
\[
\lim_{\nri}\left(\left[\frac{z^2\Phi_n(z;\mu)}{\Phi_{n+1}(z;\mu)}-z\right]\bigg|_{z=\infty}\right)=x,
\]
then
\[
\lim_{\nri}\int_{\bbC}z^j|\varphi_n(z;\mu)|^2d\mu(z)=x^j.
\]
\end{theorem}

The conclusions of Theorem \ref{equib} and Theorem \ref{equib2} also yield conclusions about the asymptotic behavior of the moments of the normalized zero counting measures as in Corollary \ref{weakzero}.

\section{Measures with Infinite Discrete Part.}\label{mass}

This section provides a proof of Corollary \ref{nvyweak} concerning measures of the form studied in \cite{Koosis}.  More specifically, we will assume $\mu$ can be written as
\[
\mu=\mu_1+\mu_2+\mu_3
\]
where $\mu_1$ satisfies $\mu_1(\bbD)=\mu_1(\bbC)$, $\mu_2$ is a measure on the unit circle of the form $w(\theta)d\theta/2\pi$ where
\[
\int_0^{2\pi}\log(w(\theta))d\theta>-\infty,
\]
and $\mu_3$ is a purely discrete measure supported on $\bbC\setminus\bard$ whose mass points $\{z_j\}_{j\in\bbN}$ satisfy the balschke condition:
\[
\sum_{j=1}^{\infty}|z_j|-1<\infty.
\]
The following theorem follows from the results in \cite{Koosis}:

\begin{theorem}[Nazarov, Volberg, $\&$ Yuditski, 2006]\label{NVY}
If $\mu$ is as above, then
\begin{enumerate}
\item
\[
\lim_{\nri}\frac{\varphi_n(z;\mu)}{\varphi_{n+1}(z;\mu)}=\frac{1}{z}
\]
when $|z|$ is sufficiently large.
\item  For every $j\in\bbN$,
\[
\lim_{\nri}\varphi_n(z_j;\mu)=0.
\]
\item
\[
\lim_{\nri}\int_{\bbC}|\varphi_n(z;\mu)|^2d\mu_1(z)=0.
\]
\end{enumerate}
\end{theorem}

Theorem \ref{NVY} easily implies that every weak asymptotic measure is a measure on $\partial\bbD$.  We can now apply Corollary \ref{equibcor} to conclude that the measures $\{|\varphi_n|^2d\mu\}_{n\in\bbN}$ converge weakly to normalized arc-length measure on the unit circle as $\nri$.

\vspace{7mm}

\noindent Brian Simanek, Vanderbilt Department of Mathematics,

\noindent\texttt{brian.z.simanek@vanderbilt.edu}


\vspace{7mm}


\begin{thebibliography}{99}


\bibitem{EGST} C. Escribano, A. Giraldo, M. Asunci\'{o}n Sastre, and E. Torrano, {\em The Hessenberg matrix and the Riemann mapping function}, Adv. Comput. Math., DOI 10.1007/s10444-012-9291-y

\bibitem{Feintuch} A. Feintuch, {\em On asymptotic Toeplitz and Hankel operators}, Operator Theory Advances and Applications, 41 (1989) 241--254.

\bibitem{GarnMar} J. Garnett and D. Marshall, {\em Harmonic Measure}, Cambridge University Press, Cambridge, 2005.


\bibitem{Kr} S. Khrushchev, {\em Schur's algorithm, orthogonal polynomials, and convergence of Wall's continued fractions in $L^2(\mathbb{T})$}, Journal of Approximation Theory, 108 (2001), 161--248.


\bibitem{MDPoly} E. Mi\~{n}a-D\'{i}az, {\em Asymptotics for polynomials orthogonal over the unit disk with respect to a positive polynomial weight}, J. Math. Anal. Appl., Vol. 372, no. 1 (2010), 306--315.

\bibitem{Koosis} F. Nazarov, A. Volberg, and P. Yuditskii, {\em  Asymptotics of orthogonal polynomials via the Koosis Theorem}, Math. Res. Lett., 13 (2006), no. 5-6, 975--983.

\bibitem{Rudin} W. Rudin, {\em  Real and Complex Analysis}, Third Edition, McGraw-Hill, Madison, WI, 1987.

\bibitem{Shift} E. B. Saff and N. Stylianopoulos, {\em Asymptotics for Hessenberg matrices for the Bergman shift operator on Jordan regions}, to appear in Complex Analysis and Operator Theory.

\bibitem{SimBlob} B. Simanek, {\em Asymptotic properties of extremal polynomials corresponding to measures supported on analytic regions}, to appear in Journal of Approximation Theory.

\bibitem{SimaRat} B. Simanek, {\em A new approach to ratio asymptotics for orthogonal polynomials}, Journal of Spectral Theory 2 (2012), no. 4, 373--395.


\bibitem{SimonWeak} B. Simon, {\em Ratio asymptotics and weak asymptotic measures for orthogonal polynomials on the real line}, J. Approx. Theory 126 (2004), 198--217.

\bibitem{OPUC1} B. Simon, {\em Orthogonal Polynomials on the Unit Circle, Part One: Classical Theory}, American Mathematical Society, Providence, RI, 2005.

\bibitem{OPUC2} B. Simon, {\em Orthogonal Polynomials on the Unit Circle, Part Two: Spectral Theory}, American Mathematical Society, Providence, RI, 2005.

\bibitem{WeakCD} B. Simon, {\em Weak convergence of CD kernels and applications}, Duke Math. J. 146 (2009), 305--330.

\bibitem{Rice} B. Simon, {\em Szeg\H{o}'s Theorem and its Descendants: Spectral Theory for $L^2$ perturbations of Orthogonal Polynomials}, Princeton University Press, Princeton, NJ, 2010.


\bibitem{Suetin} P. K. Suetin,  {\em Polynomials Orthogonal Over a Region and Bieberbach Polynomials}, American Mathematical Society, Providence, RI, 1974.


\bibitem{Corners} N. Stylianopoulos, {\em Strong asymptotics for Bergman orthogonal polynomials over domains with corners and applications}, to appear in Constructive Approximation.



\end{thebibliography}
\end{document}